\documentclass[12pt]{article}
\usepackage{amsfonts}
\usepackage{amsmath}
\usepackage{amsthm}
\usepackage{geometry}
\newcommand{\bb}[1]{\mathbb{#1}}

\newtheorem{lem}{Lemma}

\newtheorem{thm}{Theorem}
\newtheorem{conj}{Conjecture}
\newtheorem{ques}{Question}

\title{An asymptotic result concerning a question of Wilf} 
\author{Alex Zhai} 
\date{}

\begin{document}
\maketitle

\begin{abstract}
  Let $\Lambda$ be a numerical semigroup, and define $c(\Lambda) =
  \max(\bb{N} \setminus \Lambda) + 1$ and $c'(\Lambda) = \left| \{
    \lambda \in \Lambda \mid \lambda < c(\Lambda) \} \right|$. It was
  asked by Wilf whether it is always the case that

  \[ \frac{c'(\Lambda)}{c(\Lambda)} \ge \frac{1}{e(\Lambda)}, \]

  \noindent where $e(\Lambda)$ is the embedding dimension of
  $\Lambda$. We prove an asymptotic and approximate version of this
  result; in particular, we show that for a fixed positive integer $k$
  and any $\epsilon > 0$,

  \[ \frac{c'(\Lambda)}{c(\Lambda)} > \frac{1}{k} - \epsilon \]

  \noindent for all but finitely many numerical semigroups $\Lambda$
  with $e(\Lambda) = k$. To this end, we also give an explicit lower
  bound for $\frac{c'(\Lambda)}{c(\Lambda)}$ in terms of the
  multiplicity of $\Lambda$.
\end{abstract}

\section{Introduction}

A \emph{numerical semigroup} is defined to be a cofinite subsemigroup
of the non-negative integers. It is not difficult to see that any
numerical semigroup has a well-defined minimal set of generators (see,
for example, Theorem 2.7 of \cite{RosalesSanchez}). The size of this
minimal generating set is known as the \emph{embedding dimension},
which we will denote by $e(\Lambda)$.

We also define $c(\Lambda) = \max(\bb{N} \setminus \Lambda) + 1$,
known as the \emph{conductor} of $\Lambda$, and $c'(\Lambda) = \left|
  \{ \lambda \in \Lambda \mid \lambda < c(\Lambda) \} \right|$. These
quantities are closely related to the better-known quantities
$g(\Lambda) = c(\Lambda) - c'(\Lambda) + 1$, known as the
\emph{genus}, and $F(\Lambda) = c(\Lambda) - 1$, known as the
\emph{Frobenius number}. However, neither the genus nor the Frobenius
number will play a significant role in our subsequent investigations,
so we will express things in terms of $c(\Lambda)$ and $c'(\Lambda)$.

It was asked by Wilf in \cite{Wilf} whether
$\frac{c'(\Lambda)}{c(\Lambda)} \ge \frac{1}{e(\Lambda)}$ always
holds.\footnote{The original question puts the inequality in a
  slightly different but equivalent form.} In words, this says that of
the numbers less than the conductor of $\Lambda$, the proportion in
$\Lambda$ is at least the reciprocal of the embedding
dimension. Although the statement was not formulated as a conjecture
in Wilf's original paper, several authors have expressed belief that
the statement is true (\cite{BrasAmoros}, \cite{DobbsMatthews},
\cite{Kaplan}). Thus, we will refer to it as the Wilf conjecture.

\begin{conj}[Wilf]
Let $\Lambda$ be a numerical semigroup. Then,

\[ \frac{c'(\Lambda)}{c(\Lambda)} \ge \frac{1}{e(\Lambda)}. \]
\end{conj}

The above inequality has at least two families of equality cases. The
first is whenever $e(\Lambda) = 2$, which is addressed in
\cite{DobbsMatthews} (although not specifically mentioned as an
equality case). The second occurs when the minimal generators of
$\Lambda$ are $\{ k, nk + 1, nk + 2, \ldots , nk + k - 1 \}$ for some
positive integers $k$ and $n$, as mentioned in
\cite{FrobergGottliebHaggkvist}. In this case $e(\Lambda) = k$,
$c'(\Lambda) = n$, and $c(\Lambda) = nk$.

The first significant effort towards proving the Wilf conjecture was
made by Dobbs and Matthews in \cite{DobbsMatthews}. They prove the
inequality in the cases $e(\Lambda) \le 3$, $c(\Lambda) \le 21$,
$c'(\Lambda) \le 4$, and $\frac{c(\Lambda)}{4} \le c'(\Lambda)$. More
recently, Kaplan also proved the inequality when there are certain
restrictions relating to the \emph{multiplicity} of $\Lambda$, which
is defined to be $\min(\Lambda \setminus \{ 0 \})$ and is denoted by
$m$ or $m(\Lambda)$. In particular, he showed in \cite{Kaplan} that
the inequality holds if $3m(\Lambda) > 2(c(\Lambda) - c'(\Lambda))$ or
$2m \ge c(\Lambda)$.\footnote{The original paper formulated these
  inequalities in terms of the genus and the Frobenius number, but in
  order to use consistent notation, we have made the appropriate
  substitutions with $c(\Lambda)$ and $c'(\Lambda)$.} In addition, the
Wilf conjecture has been numerically verified by Bras-Amor\'os in
\cite{BrasAmoros} for a large number of numerical semigroups.

However, the author is not aware of any previous work that addresses
the Wilf conjecture for a general numerical semigroup. One of the
approaches in \cite{DobbsMatthews} considers the \emph{type} of a
numerical semigroup $\Lambda$, defined to be the number of integers $x
\not \in \Lambda$ such that $x + \lambda \in \Lambda$ for any non-zero
$\lambda \in \Lambda$. Denoting the type by $t(\Lambda)$, it is known
that $\frac{c'(\Lambda)}{c(\Lambda)} \ge \frac{1}{t(\Lambda) + 1}$
(see \cite{RosalesSanchez}, Proposition 2.26). Thus, to prove the Wilf
conjecture, it would suffice to show that $t(\Lambda) + 1 \le
e(\Lambda)$. Unfortunately, this approach fails even for $e(\Lambda) =
4$ (see \cite{DobbsMatthews}, Remark 2.13).

There are two main results in this paper. The first result establishes
a weakened version of the Wilf conjecture in terms of the multiplicity
$m(\Lambda)$.

\begin{thm} \label{main theorem} 

  For any numerical semigroup $\Lambda$,

  \[ c'(\Lambda) \ge \frac{c(\Lambda)}{e(\Lambda)} - \frac{(m(\Lambda)
    - 1)(e(\Lambda) - 2)}{2 e(\Lambda)}. \]

\end{thm}

\noindent Theorem \ref{main theorem} leads to the second main result,
which establishes an asymptotic version of the Wilf conjecture.

\begin{thm} \label{asymptotic consequence}

  Fix a positive integer $k$. Then, for any $\epsilon > 0$, we find
  that

  \[ \frac{c'(\Lambda)}{c(\Lambda)} > \frac{1}{k} - \epsilon \]

  \noindent for all but finitely many numerical semigroups $\Lambda$
  satisfying $e(\Lambda) = k$.
\end{thm}

The remainder of the paper is organized as follows. In Section
\ref{sec: apery}, we provide some background on Ap\'ery sets and their
connection to $c(\Lambda)$ and $c'(\Lambda)$. In Section \ref{sec:
  subset bound}, we prove a lemma concerning subsets of $\bb{Z}^d$ in
preparation for the proofs of the main results in Section \ref{sec:
  main theorems}. Finally, in Section \ref{sec: conclusions} we
discuss our approach and pose some questions for future
investigation.

\section{Ap\'ery sets} \label{sec: apery}

The first step in proving our main results is to compute $c'(\Lambda)$
in terms of the Ap\'ery set (as seen in, for example, Chapter 1,
Section 2 of \cite{RosalesSanchez}\footnote{The authors define more
  generally the Ap\'ery set \emph{of $n$} in $\Lambda$, for any
  positive integer $n$. For our purposes, we will always assume $n =
  m$.}). For a numerical semigroup $\Lambda$ with multiplicity $m$,
the \emph{Ap\'ery set} of $\Lambda$ is defined to be the set $\{ a_0,
a_1, \ldots , a_{m - 1} \}$, where $a_i$ is the smallest element in
$\Lambda$ congruent to $i$ modulo $m$. We will denote the Ap\'ery set
of $\Lambda$ by $A(\Lambda)$.

Since $m \in \Lambda$, for each $i$, the set of elements of $\Lambda$
congruent to $i$ modulo $m$ is exactly $\{ a_i, a_i + m, a_i + 2m ,
\ldots \}$. Note that this gives an alternate characterization of
$A(\Lambda)$ as the set $\{ a \in \Lambda \mid a - m \not \in \Lambda
\}$.

By separately examining the elements of $\Lambda$ congruent to each
residue modulo $m$, we may compute $c'(\Lambda)$ in terms of
$A(\Lambda)$.

\begin{lem} \label{lem: c' formula} 

  \[ c'(\Lambda) = \frac{m - 1}{2} + \sum_{i = 0}^{m - 1}
  \frac{c(\Lambda) - a_i}{m}. \]

\end{lem}
\begin{proof}
  For each $i$, let $\sigma(i)$ be the (unique) element between $0$
  and $m - 1$ inclusive which is congruent to $i - c(\lambda)$ modulo
  $m$. The elements of $\Lambda$ congruent to $i$ modulo $m$ and less
  than $c(\Lambda)$ are

  \[ \left\{ a_i, a_i + m, \ldots , c(\Lambda) + \sigma(i) - m
  \right\}. \]

  \noindent This last element is equal to $a_i + \left(
    \frac{c(\Lambda) + \sigma(i) - a_i}{m} - 1 \right)m$. Thus, there
  are $\frac{c(\Lambda) + \sigma(i) - a_i}{m}$ elements of $\Lambda$
  less than $c(\Lambda)$ congruent to $i$ modulo $m$. Summing over all
  $i$ and noting that $c'(\Lambda)$ is the total number of elements of
  $\Lambda$ less than $c(\Lambda)$, we obtain 

  \[ c'(\Lambda) = \sum_{i = 0}^{m - 1} \frac{c(\Lambda) + \sigma(i) -
    a_i}{m}. \]  

  \noindent Finally, note that $\sigma(0), \sigma(1), \ldots ,
  \sigma(m - 1)$ is a permutation of $0, 1, \ldots , m - 1$, so we have

  \[ \sum_{i = 0}^{m - 1} \frac{\sigma(i)}{m} = \sum_{i = 0}^{m - 1}
  \frac{i}{m} = \frac{m - 1}{2}. \]

  \noindent Plugging this in yields the desired equation.
\end{proof}

\noindent We record one additional lemma, which is trivial but useful
to keep in mind.

\begin{lem} \label{lem: a <= c + m - 1}
  For each $a \in A(\Lambda)$, $a \le c(\Lambda) + m - 1$.
\end{lem}
\begin{proof}
  If $a \in A(\Lambda)$, then $a - m \not \in \Lambda$. By the
  definition of $c(\Lambda)$, this implies that $a - m \le c(\Lambda)
  - 1$. Rearranging yields the result.
\end{proof}

We can treat Lemma \ref{lem: a <= c + m - 1} as a constraint on the
elements of $A(\Lambda)$, and our task is to give a lower bound on
$c'(\Lambda)$, as given by Lemma \ref{lem: c' formula}. Of course, the
set $A(\Lambda)$ is far from arbitrary, and to obtain a good bound, we
must use some of its properties relating to the semigroup structure of
$\Lambda$. 

Suppose that $\Lambda$ has minimal generators $m = g_1, g_2, \ldots ,
g_k$. Then, each element $a \in A(\Lambda)$ can be written as $\sum_{i
  = 1}^k r_ig_i$ for some (not necessarily unique) non-negative
integers $r_i$. Furthermore, since $a - m \not \in \Lambda$, it
follows that $r_1 = 0$. Hence, each element of $A(\Lambda)$ can be
considered as a point $(r_2, \ldots , r_k) \in \bb{Z}^{k - 1}$. In the
next section, we prove an inequality on subsets of $\bb{Z}^d$, which
will then be applied to $A(\Lambda)$ to prove Theorem \ref{main
  theorem}.

\section{A bound on subsets of $\bb{Z}^d$} \label{sec: subset bound}

Let $O^d \subset \bb{Z}^d$ be the set of $d$-tuples of non-negative
integers. Then, we have the following lemma.

\begin{lem} \label{inequality} Let $y_1, y_2, \ldots, y_d$, and $C$ be
  positive reals. Define a function $\pi(x_1, \ldots , x_d) = \sum_{i
    = 1}^d x_iy_i$. Suppose that $S \subset O^d$ is a (finite) set
  such that $\pi({\bf s}) \le C$ for all ${\bf s} \in S$.

  Furthermore, suppose that $O^d + (O^d \setminus S) \subset (O^d
  \setminus S)$, where the sum $A + B$ of two sets $A, B \subset
  \bb{Z}^d$ is defined to be the set $\{ a + b \mid a \in A, b \in B
  \}$.

  Then, the following inequality holds:

  \[ (d + 1) \sum_{{\bf s} \in S} (C - \pi({\bf s})) \ge C |S|. \]
\end{lem}
\begin{proof} For each $j$, define the map $f_j: \bb{Z}^d \to \bb{R}$
  by $(x_1, \ldots , x_d) \mapsto x_jy_j + \pi(x_1, \ldots ,
  x_d)$. Next, fix a value of $j$, and fix values $x_i'$ for $i \ne
  j$. Define ${\bf s}_x = (x'_1, \ldots , x'_{j-1}, x, x'_{j+1},
  \ldots , x_d)$. Also, let $D$ be the largest integer such that ${\bf
    s}_D \in S$. 

  Note that for any non-negative $x \le D$, if ${\bf s}_x \in O^d
  \setminus S$, then since ${\bf s}_D - {\bf s}_x \in O^d$, it follows
  from the hypothesis $O^d + (O^d \setminus S) \subset (O^d \setminus
  S)$ that ${\bf s}_D \in O^d \setminus S$, a contradiction. Hence,
  ${\bf s}_x \in S$ for all non-negative $x \le D$. Summing these
  under $f_j$ and noting that $f_j({\bf s}_x) + f_j({\bf s}_{D - x}) =
  \pi({\bf s}_{2x}) + \pi({\bf s}_{2D - 2x}) = 2 \pi({\bf s}_D) \le
  2C$, we have

  \[ \sum_{x = 0}^D f_j({\bf s}_x) = \frac{1}{2} \sum_{x = 0}^D
  (f_j({\bf s}_x) + f_j({\bf s}_{D - x})) \]
  \[ \le \frac{1}{2} \sum_{x = 0}^D 2C = \sum_{x = 0}^D C. \]

  \noindent Summing over all choices of the $x'_i$, we find that
  
  \[ \sum_{{\bf s} \in S} f_j({\bf s}) \le \sum_{{\bf s} \in S} C. \]

  \noindent We can then sum this over all $j$, and noting that
  $\sum_{j = 1}^d f_j({\bf s}) = (d + 1) \pi({\bf s})$, we obtain

  \[ \sum_{{\bf s} \in S} (d + 1) \pi({\bf s}) \le d \sum_{{\bf s} \in
    S} C \]
  \[ C |S| = \sum_{{\bf s} \in S} C \le (d + 1) \sum_{{\bf s} \in S}
  (C - \pi({\bf s})), \]

  \noindent as desired.
\end{proof}

\section{Proofs of the main theorems } \label{sec: main theorems}

To prove Theorem \ref{main theorem}, we will translate elements of
$\Lambda' = \{ \lambda \in \Lambda \mid \lambda < c(\Lambda) \}$ into
a set $S$ as in Lemma \ref{inequality}.

We use the notation of Section \ref{sec: apery}, where $m$ is the
multiplicity of $\Lambda$, $a_i$ denotes the smallest element of
$\Lambda$ congruent to $i$ modulo $m$, and $A(\Lambda) = \{ a_0,
\ldots , a_{m - 1} \}$.

\begin{proof}[Proof of Theorem \ref{main theorem}]
  Let $k = e(\Lambda)$, and let $m = g_1 < g_2 < \cdots < g_k$ be the
  minimal generators of $\Lambda$.

  We will apply Lemma \ref{inequality} with $d = k - 1$ and $y_1 =
  g_2, y_2 = g_3, \ldots , y_{k - 1} = g_k$. Thus, using the notation
  of Lemma \ref{inequality}, $\pi(x_1, \ldots , x_{k - 1}) = \sum_{i =
    1}^{k - 1} x_ig_{i + 1}$. Now, recall that any element $a \in
  A(\Lambda)$ can be expressed as a sum $\sum_{i = 2}^k r_ig_i$, where
  $r_i \ge 0$ for all $i$. In other words, there exists some ${\bf x}
  \in O^{k - 1}$ such that $\pi({\bf x}) = a$.

  We are interested in a particular choice of such ${\bf x}$. We
  define a lexicographical ordering on $O^{k - 1}$ by considering the
  $k - 1$ coordinates as a sequence of $k - 1$ numbers, and we say an
  element ${\bf x} \in O^{k - 1}$ is \emph{lexicographically minimal}
  if it comes lexicographically before any other element ${\bf y} \in
  O^{k - 1}$ for which $\pi({\bf x}) = \pi({\bf y})$.
  
  For each $i$, define ${\bf s}_i$ to be the lexicographically minimal
  element of $O^{k - 1}$ such that $\pi({\bf s}_i) = a_i$. We claim
  that $S = \{ {\bf s}_0, {\bf s}_1, \ldots , {\bf s}_{m - 1} \}$
  satisfies the hypotheses of Lemma \ref{inequality} with $C =
  c(\Lambda) + m - 1$. By Lemma \ref{lem: a <= c + m - 1}, we have
  $\pi(s_i) = a_i \le c(\Lambda) + m - 1 = C$, so the first condition
  is satisfied. To verify the second condition, consider any ${\bf x}
  \in O^{k - 1} \setminus S$ and ${\bf x}' \in O^{k - 1}$. We claim
  that ${\bf x} + {\bf x}' \not \in S$.

  Since ${\bf x} \not \in S$, either $\pi({\bf x}) \not \in
  A(\Lambda)$, or ${\bf x}$ is not lexicographically minimal. In the
  case that $\pi({\bf x}) \not \in A(\Lambda)$, recall that
  $A(\Lambda) = \{ a \in \Lambda \mid a - m \not \in \Lambda \}$. We
  thus know that $\pi({\bf x}) - m \in \Lambda$. In addition,
  $\pi({\bf x}')$ is a non-negative integer combination of $g_2,
  \ldots , g_k$, so $\pi({\bf x}') \in \Lambda$. Therefore, $\pi({\bf
    x} + {\bf x}') - m = \left( \pi({\bf x}) - m \right) + \pi({\bf
    x}') \in \Lambda$. Hence, we find that $\pi({\bf x} + {\bf x}')
  \not \in A(\Lambda)$, and so ${\bf x} + {\bf x}' \not \in S$.

  In the case that ${\bf x}$ is not lexicographically minimal, let
  ${\bf x}_0 \in O^k$ be an element which comes lexicographically
  before ${\bf x}$ such that $\pi({\bf x}_0) = \pi({\bf x})$. Then,
  ${\bf x}_0 + {\bf x}'$ comes lexicographically before ${\bf x} +
  {\bf x}'$, while $\pi({\bf x}_0 + {\bf x}') = \pi({\bf x} + {\bf
    x}')$. It follows that ${\bf x} + {\bf x}'$ is also not
  lexicographically minimal, so ${\bf x} + {\bf x}' \not \in S$.

  This establishes our claim, and the hypotheses of Lemma
  \ref{inequality} are satisfied. Setting $C = c(\Lambda) + m - 1$, we
  obtain

  \[ Cm = C|S| \le k \sum_{{\bf s} \in S} (C - \pi({\bf s})) \]
  \[ \frac{C}{k} \le \frac{1}{m} \sum_{i = 0}^{m - 1} (C - \pi({\bf
    s}_i)). \]

  \noindent We can apply this inequality to the formula in Lemma
  \ref{lem: c' formula} to find that

  \[ c'(\Lambda) = \frac{m - 1}{2} + \sum_{i = 0}^{m - 1}
  \frac{c(\Lambda) - a_i}{m} \]
  \[ = \frac{m - 1}{2} + \sum_{i = 0}^{m - 1} \frac{C - m + 1 -
    \pi({\bf s}_i)}{m} \]
  \[ \ge \frac{C}{k} + \frac{m - 1}{2} - \sum_{i = 0}^{m - 1} \frac{m
    - 1}{m} \]
  \[ = \frac{C}{k} - \frac{m - 1}{2} \]
  \[ = \frac{c(\Lambda)}{k} - \frac{(m - 1)(k - 2)}{2k}, \]

  \noindent which proves the theorem upon substituting $e(\Lambda) =
  k$.
\end{proof}

The inequality of Theorem \ref{main theorem} is weaker than the Wilf
conjecture, but it implies the asymptotic result of Theorem
\ref{asymptotic consequence}, because in some sense, $m(\Lambda)$ is
much smaller than $c(\Lambda)$ for ``most'' numerical semigroups
$\Lambda$. We make this precise below.

\begin{lem} \label{lem: m/c --> 0} 

  For any $\epsilon > 0$ and fixed positive integer $k$, there are
  only finitely many numerical semigroups $\Lambda$ such that
  $e(\Lambda) = k$, and $\frac{m(\Lambda)}{c(\Lambda)} > \epsilon$.

\end{lem}
\begin{proof}
  Let $g_1, \ldots , g_k$ be the minimal generators of $\Lambda$, and
  define $\pi({\bf x}) = \sum_{i = 1}^k x_ig_i$ for ${\bf x} = (x_1,
  \ldots , x_k) \in O^k$ (note that this differs slightly from the
  definition of $\pi$ used in proving Theorem \ref{main theorem},
  which took the sum of $x_ig_{i + 1}$). Suppose that
  $\frac{m(\Lambda)}{c(\Lambda)} > \epsilon$. We will show that
  $c(\Lambda) \le \frac{2^k}{\epsilon^k}$.

  For any ${\bf x} = (x_1, \ldots , x_k) \in O^k$, if $x_i \ge
  \frac{2}{\epsilon}$ for any $i$, then

  \[ \pi({\bf x}) \ge x_ig_i \ge \frac{2m}{\epsilon} > 2c(\Lambda). \]

  \noindent Therefore, there are at most $\frac{2^k}{\epsilon^k}$
  possible values of ${\bf x} \in O^k$ for which $\pi({\bf x}) \le
  2c(\Lambda)$.

  Any element $\lambda \in \Lambda$ can be expressed as a sum $\lambda
  = \sum_{i = 1}^k r_ig_i$, where the $r_i$ are non-negative
  integers. Letting ${\bf r} = (r_1, \ldots , r_k)$, we find that
  $\pi({\bf r}) = \lambda$. Thus, $\pi(O^k) \supset \Lambda$. This
  implies that

  \[ | \Lambda \cap [0, 2c(\Lambda)] | \le | \pi^{-1}( \Lambda \cap
  [0, 2c(\Lambda)] ) | \le \frac{2^k}{\epsilon^k}. \]

  However, by the definition of $c(\Lambda)$, we know that $\Lambda$
  contains all the numbers between $c(\Lambda)$ and $2c(\Lambda)$, and
  so $| \Lambda \cap [0, 2c(\Lambda)] | \ge c(\Lambda)$. Therefore,
  $c(\Lambda) \le \frac{2^k}{\epsilon^k}$, meaning that there are
  finitely many possible values of $c(\Lambda)$.

  For each possible value of $c(\Lambda)$, it is clear that there are
  only finitely many possibilities for $\Lambda$. Thus, only finitely
  many $\Lambda$ satisfy the specified conditions.
\end{proof}

\noindent Theorem \ref{asymptotic consequence} is then an immediate
consequence of Lemma \ref{lem: m/c --> 0} and Theorem \ref{main
  theorem}.

\begin{proof}[Proof of Theorem \ref{asymptotic consequence}.]

  Let $\Lambda$ be a numerical semigroup satisfying $e(\Lambda) =
  k$. Note that if $k = 1$, then in order for $\Lambda$ to be
  cofinite, $\Lambda$ must consist of all the non-negative
  integers. Recall also that if $k = 2$, then
  $\frac{c'(\Lambda)}{c(\Lambda)} = \frac{1}{k}$. Thus, Theorem
  \ref{asymptotic consequence} clearly holds in these two cases, and
  we may henceforth assume that $k > 2$.

  Take $\epsilon' = \frac{2k \epsilon}{k - 2}$. Then, according to
  Lemma \ref{lem: m/c --> 0}, by excluding only finitely many
  numerical semigroups, we may assume that
  $\frac{m(\Lambda)}{c(\Lambda)} \le \epsilon'$. Then, by Theorem
  \ref{main theorem},

  \[ c'(\Lambda) \ge \frac{c(\Lambda)}{k} - \frac{(m(\Lambda) - 1)(k -
    2)}{2k}, \]

  \noindent which rearranges to
  
  \[ \frac{c'(\Lambda)}{c(\Lambda)} \ge \frac{1}{k} - \frac{m(\Lambda)
    - 1}{c(\Lambda)} \cdot \frac{k - 2}{2k} \]
  \[ > \frac{1}{k} - \frac{m(\lambda)}{c(\Lambda)} \cdot \frac{\epsilon}{\epsilon'} \]
  \[ \ge \frac{1}{k} - \epsilon. \]
\end{proof}

\section{Concluding remarks} \label{sec: conclusions}

Having proven an asymptotic version of the Wilf conjecture, it is
natural to ask whether the techniques used may be adapted to prove the
inequality exactly. Recalling the proof of Theorem \ref{main theorem},
the only place where an inequality is established is through the
application of Lemma \ref{inequality}. 

An astute reader may have noticed, however, that equality in Lemma
\ref{inequality} can occur. For example, if $y_1 = y_2 = \cdots = y_d
= 1$ and $C$ is any positive integer, then it is not hard to check
that $S = \{ (x_1, \ldots , x_d) \in O^d \mid x_1 + \cdots + x_d \le C
\}$ gives equality. In fact, up to scaling of the $y_i$ and $C$, it is
fairly straightforward to show that this is the only equality case.

Thus, any direct approach towards strengthening Theorem \ref{main
  theorem} would likely focus on how the application of Lemma
\ref{inequality} in the proof deviates from the equality case. Recall
that in the proof of Theorem \ref{main theorem}, the $y_i$ were taken
to be the minimal generators of $\Lambda$, excluding $m$. Since the
minimal generators are distinct, it is certainly impossible to exactly
achieve the equality case described above, where all the $y_i$ are
equal.

A sort of approximation to the equality case of Lemma \ref{inequality}
occurs when the minimal generators excluding $m$ are very close to
each other; this corresponds to the equality case of the Wilf
conjecture in which the minimal generators of $\Lambda$ are $m$, $nm +
1, nm + 2, \ldots , nm + m - 1$, for any $n > 1$. It would be
interesting to investigate the deviation from equality more precisely.

It is also natural to ask whether the asymptotic results we have
obtained are, in some sense, sharp. Note that for any fixed embedding
dimension $k$, the equality case just described actually gives an
infinite family of equality cases if we take $m = k$. However, the
requirement $m = k$ is rather restrictive, because it implies that
every element of the Ap\'ery set is also a minimal generator. One
might ask whether larger families of equality or near-equality cases
exist. For any $M$ and $k$, let $F(M, k)$ denote the infimum of
$\frac{c'(\Lambda)}{c(\Lambda)}$ over all numerical semigroups
$\Lambda$ satisfying $e(\Lambda) = k$ and $m(\Lambda) > M$. Then, we
ask the following question.

\begin{ques}
For each $k$, what is the value of

\[ \lim_{M \rightarrow \infty} F(M, k), \]

\noindent and in particular, is it equal to $\frac{1}{k}$?
\end{ques}

Although there has been considerable recent attention on the study of
numerical semigroups according to their genus or Frobenius number
(e.g. \cite{Kaplan}, \cite{BlancoSanchezPuerto}), there have been few
prior results relating these quantities to the embedding dimension
except in certain special cases. We hope that this work will enable a
better understanding of how embedding dimension fits into the
picture. In addition, we believe that the results of this paper
provide further evidence that the Wilf conjecture is true and hope
that they may be strengthened to prove the conjecture in full.

\section{Acknowledgements} \label{sec: acknowledgements}

The author is indebted to Ricky Liu for helpful discussions and the
idea for proving Lemma \ref{inequality}, and the author gratefully
acknowledges Nathan Kaplan for suggestions to improve the exposition
and organization of the paper.

This work was done as part of the University of Minnesota Duluth REU
program, funded by the National Science Foundation (Award DMS 0754106)
and the National Security Agency (Award H98230-06-1-0013). The author
would also like to thank Joseph Gallian for organizing the program,
proofreading the paper, and providing various forms of support along
the way. Finally, the author would like to thank Maria Monks and
Nathan Pflueger for advising and feedback on oral presentations of the
material here that the author gave before the time of writing.


\begin{thebibliography}{10}

\bibitem{BlancoSanchezPuerto} V. Blanco, P. A. Garc\'ia-S\'anchez, and
  J. Puerto, \textit{Computing the number of numerical semigroups
    using generating functions}, arXiv:0901.1228v3 (2009).
\bibitem{BrasAmoros} M. Bras-Amor\'os, \textit{Fibonacci-like behavior
    of the number of numerical semigroups of a given genus}, Semigroup
  Forum \textbf{76} (2008), no. 2, 379-384.
\bibitem{DobbsMatthews} D. Dobbs and G. Matthews, \textit{On a
    question of Wilf concerning numerical semigroups}, Focus on
  commutative rings research, 193-202, New York: Nova Science
  Publishers, 2006.
\bibitem{FrobergGottliebHaggkvist} R. Fr\"oberg, C. Gottlieb, and
  R. H\"aggkvist, \textit{On numerical semigroups}, Semigroup Forum
  \textbf{35} (1987), no. 1, 63-83.
\bibitem{Kaplan} N. Kaplan, \textit{Counting numerical semigroups by
    genus and some cases of a question of Wilf}, preprint.
\bibitem{RosalesSanchez} J. C. Rosales and P. A. Garc\'ia-S\'anchez,
  \textit{Numerical Semigroups}, New York: Springer, 2009.
\bibitem{Wilf} H. Wilf, \textit{A circle-of-lights algorithm for the
    ``money-changing problem''}, American Mathematical Monthly
  \textbf{85} (1978), no. 7, 562-565.

\end{thebibliography}
\end{document}